\documentclass{amc}
 \usepackage{amsmath}
 \usepackage{epsfig} 
 \usepackage{graphics} 

\usepackage{amsfonts}
\usepackage{amssymb}
\usepackage{mathrsfs}
\usepackage{amsthm}
\usepackage{xfrac}
\usepackage{nicefrac}
\usepackage{lmodern}
\usepackage{fix-cm}
\usepackage{stmaryrd}
\usepackage{enumitem}
\usepackage{fourier}

\usepackage[utf8]{inputenc}

\def\currentvolume{X}
 \def\currentissue{X}
  \def\currentyear{200X}
   
    \def\ppages{X--XX}
    \def\DOI{10.3934/xx.xx.xx.xx}

\newtheorem{theorem}{Theorem}
\newtheorem{corollary}{Corollary}

\newtheorem{lemma}{Lemma}
\newtheorem{proposition}{Proposition}

\theoremstyle{definition}
\newtheorem{definition}{Definition}
\newtheorem{remark}{Remark}
\newtheorem{example}[theorem]{Example}

\theoremstyle{algorithm}
\newtheorem{algorithm}[theorem]{Algorithm}

\newcommand{\N}{\mathbb{N}}
\newcommand{\ri}{\mathrm{ri}^W}
\newcommand{\w}{\mathrm{w}}
\newcommand{\ev}{\mathrm{ev}}
\newcommand{\HP}{\mathrm{HP}}
\newcommand{\lcm}{\mathrm{lcm}}
\newcommand{\lt}{\mathrm{lt}}
\newcommand{\lm}{\mathrm{lm}}
\newcommand{\rem}{\mathrm{rem}}
\newcommand{\Supp}{\mathrm{Supp}}

\title[Hilbert quasi-polynomial for order domains and application to coding theory]
      {Hilbert quasi-polynomial for order domains and application to coding theory}

\author[Carla Mascia, Giancarlo Rinaldo, Massimiliano Sala]{}


\email{carla.mascia@unitn.it, giancarlo.rinaldo@unitn.it, maxsalacodes@gmail.com}

\thanks{This research was partially funded by the Italian Ministry of Education, Universities and Research, with the project PRIN 2015TW9LSR  "Group theory and applications".}

\begin{document}
\maketitle

\centerline{\scshape Carla Mascia, Giancarlo Rinaldo, Massimiliano Sala}
\medskip
{\footnotesize
 \centerline{Università degli Studi di Trento}
  \centerline{Trento, Italy}
   }



\medskip

 \centerline{(Communicated by Aim Sciences)}
 \medskip

\begin{abstract}
We present an application of Hilbert quasi-polynomials to order domains, allowing the effective check of the second order-domain condition in a direct way.
We also provide an improved algorithm for the computation of the related Hilbert quasi-polynomials. This allows to identify order domain codes more easily.
\end{abstract}

\section{Introduction}

Fundamental algebraic invariants of a standard-graded ring can be deduced from its Hilbert-Poincar\'e series and Hilbert polynomial. The Hilbert quasi-polynomial generalizes Hilbert polynomial from the standard grading case to more generalized weighted grading, but until \cite{CGC-alg-art-cabmasc16} no effective algorithms for its computation were known. 

Apart from its natural ideal-theoretical application, we believe that more practical applications can arise from its use, e.g. for the study of order domains and the related codes.
We consider order domains as pairs of a quotient ring $R/I$ and a generalized weighted degree ordering which satisfy some conditions that depend largely on the weights of monomials under the Hilbert staircase of the ideal $I$ of $R$.

An important research area in coding theory is that of algebraic geometric codes, which are known to achieve near-optimal performance since the seminal Goppa paper (\cite{CGC-cd-art-goppa2}). A class of algebraic geometric codes that have received
a lot of recent attention is formed by the so-called order domain codes (\cite{CGC-cd-art-GeilPell},\cite{CGC-cd-book-Geil08}). These codes are defined by evaluating a polynomial vector space at points of a variety which is definided starting from an order domain, and as such they form a subclass of the so-called affine variety codes (\cite{CGC-cd-art-lax}, \cite{CGC-cd-prep-manumaxchiara12}).  In \cite{CGC-cd-book-T.HPell}, H{\o}holdt, van Lint and Pellikaan introduced order domains for the first time. They showed how to deal with one-point geometric Goppa codes in the language of order domain theory. In \cite{CGC-cd-art-AndeGeil}, Andersen and Geil give an improved bound on the minimum distance of one-point geometric Goppa codes, an improved construction of one-point geometric Goppa codes and a generalization of the bound and the improved construction to algebraic structures of higher transcendence degrees. To derive these results, they consider first the problems in the most general set-up described by Miura in \cite{CGC-cd-art-Miura} and by Miura and Matsumoto \cite{matsumoto2000feng}.

In this paper we present our application of Hilbert quasi-polynomials to order domains, which consists in verifying the order domain conditions in a direct way, once the quasi-polynomial has been computed for a related quotient ring. This computation is effective thanks to our algorithm, that improves on that
of \cite{CGC-alg-art-cabmasc16} and can be specialized to the order domain case.

The remainder of this paper is organized as follows:
\begin{itemize}
\item[Section \ref{prel}:] here we provide some notation, preliminaries and known results on Hilbert functions, Hilbert quasi-polynomials, order domains and order domain codes;
\item[Section \ref{appl}:] in this section we present some results on Hilbert quasi-polynomials, which lead to improvements in their effective calculation, and our main results Corollary \ref{cor1} and Algorithm \ref{algo1}, which allow to decide effectively if the second order-domain condition is actually satisfied.
 \item[Section \ref{maximal}:] we show some examples of application of our results. In particular, the most important family of affine-variety codes is that of codes coming from maximal curves, since these codes have large length. In this section we specialize our previous results to this family, showing some actual practical cases that can be solved easily.
\item[Section \ref{concl}:] in this section we draw our conclusions and point at possible future improvements.
\end{itemize}

\section{Preliminaries} \label{prel}
In this section we fix some notation and recall some known results. We denote by $\mathbb{N}_+$ the set of positive integers, by $\mathbb{K}$ a field, by $R:=\mathbb{K}[x_1,\dots, x_n]$ the polynomial ring in $n$ variables over $\mathbb{K}$, and by $\mathcal{M} = \mathcal{M}(X)$ the set of all monomials in the variables $x_1, \dots, x_n$.
We assign a weight $w_i \in \mathbb{N}_+^r$ to each variable $x_i$, i.e. if $X^{\mathbf{\alpha}} = x_1^{\alpha_1}\cdots x_n^{\alpha_n} \in \mathcal{M}$
\[
\w(X^{\mathbf{\alpha}}) := \alpha_1 w_1  + \cdots + \alpha_n w_n 
\] 
Let $\prec_{\mathbb{N}_+^r}$ and $\prec_{\mathcal{M}}$ be monomial orderings on $\mathbb{N}_+^r$ and $\mathcal{M}$ respectively, and let $W:=[w_1, \dots, w_n] \in (\mathbb{N}_+^r)^n$ be the vector of the variable weights. The \textit{generalized weighted degree ordering $\prec_{W}$ defined from $W$, $\prec_{\mathbb{N}_+^r}$ and $\prec_{\mathcal{M}}$} is the ordering given by $X^{\mathbf{\alpha}} \prec_{W} X^{\mathbf{\beta}}$, with $X^{\mathbf{\alpha}}, X^{\mathbf{\beta}} \in \mathcal{M}$, if either 
\[
\w(X^{\mathbf{\alpha}}) \prec_{\mathbb{N}_+^r} \w(X^{\mathbf{\beta}}) \; \text{ or }
\]
\[
\w(X^{\mathbf{\alpha}}) = \w(X^{\mathbf{\beta}}) \; \text{ and } X^{\mathbf{\alpha}} \prec_{\mathcal{M}} X^{\mathbf{\beta}} \;
\]
Given $f \in R$ and $\prec_W$, $\lm(f)$ (resp. $\lt(f)$) stands for the leading monomial (resp. leading term) of $f$ w.r.t $\prec_W$. Let $I \subseteq R$ be an ideal of $R$, then we denote by $\bar{I} := \mathrm{in}_{\prec_W}(I)$ the \textit{initial ideal} of $I$, which is the ideal generated by $\{ \lt(f) \mid f \in I\}$, and by $\mathcal{N}_{\prec_W}(I)$ the \textit{Hilbert staircase} of $I$, which is the set of all monomials that are not leading monomials of any polynomial in $I$.
In the remainder of this paper, we suppose $w_i \in \mathbb{N}_+$, unless specified otherwise. When $W =[1, \dots, 1]$, the grading is called \textit{standard}. The pair $(R,\prec_{W})$ stands for the polynomial ring with the generalized weighted degree ordering $\prec_{W}$. A polynomial $f \in (R,\prec_{W})$ is called \textit{$W$-homogeneous} if all nonzero terms of $f$ have the same weight. An ideal $I \subseteq (R,\prec_{W})$ is called \textit{$W$-homogeneous} if it is generated by a set of $W$-homogeneous polynomials.

\subsection{Order domains}

In this section we are going to give the notion of order domain, which represents a relatively new tool in the study of algebraic geometric codes. We refer to \cite{CGC-cd-inbook-D1} and \cite{CGC-cd-art-GeilPell}.

Recall that an $\mathbb{K}$-algebra is a commutative ring with a unit that contains $\mathbb{K}$ as a unitary subring. The standard example of an $\mathbb{K}$-algebra is $R = \mathbb{K}[x_1,\dots, x_n]$.

\begin{definition}
Let $\Gamma$ be a semigroup and $\prec$ a well-ordering. An \textbf{order function} on an $\mathbb{K}$-algebra $R$ is a surjective function 
\[
\rho: R \rightarrow \Gamma \cup \{-\infty \}
\]
such that the following conditions hold
\begin{enumerate}[label={(O.\arabic*)}]
\item $\rho (f) = - \infty$ if and only if $f=0$
\item $\rho(af)=\rho(f)$ for all nonzero $a \in \mathbb{K}$ 
\item $\rho(f+g) \leq \max\{ \rho(f), \rho(g)\}$
\item If $\rho(f) < \rho(g)$ and $h \neq 0$, then $\rho(fh) < \rho(gh)$
\item If $f$ and $g$ are nonzero and $\rho(f) = \rho(g)$, then there exists a nonzero $a \in \mathbb{F}$ such that $\rho(f - ag) < \rho(g)$
\end{enumerate}
for all $f,g,h \in R$.
\end{definition}

\begin{definition}
Let $R$ be an $\mathbb{K}$-algebra, $(\Gamma, \prec)$ a well-order and $\rho: R \rightarrow \Gamma \cup \{-\infty\}$ an order function. Then $(R, \rho, \Gamma)$ is called an \textbf{order structure} and $R$ an \textbf{order domain} over $\mathbb{K}$.
\end{definition}

All order functions relevant in coding theory are actually weight functions.

\begin{definition}
Let $R$ be an $\mathbb{K}$-algebra. A \textbf{weight function} on $R$ is an order function on $R$ that satisfies furthermore 
\begin{enumerate}
\item[(O.6)] $\rho (fg) = \rho(f) + \rho(g)$
\end{enumerate}
for all $f,g \in R$. Here $- \infty + n = -\infty$ for all $n \in \N$.

\end{definition}

Order domains and weight functions represent helpful tools to construct a large class of algebraic geometric codes. Actually, one can only use Groebner basis theoretical methods for constructing order domains and weight functions, instead of their formal definition, as showed by the following theorem.

\begin{theorem}\label{theoorderdomain}\textnormal{\cite{CGC-cd-inbook-D1}}
Let $I$ be an ideal in $R=\mathbb{K}[x_1, \dots, x_n]$ and assume $\mathcal{G}$ is a Groebner basis for $I$ with respect to a generalized weighted degree ordering $\prec_{W}$, with $W \subseteq \mathbb{N}_+^r$. Suppose that 
\begin{itemize}
\item[(C1)] any $g \in \mathcal{G}$ has exactly two monomials of highest weight in its support;
\item[(C2)] no two monomials in the staircase $\mathcal{N}_{\prec_W}(I)$ of $I$ are of the same weight.
\end{itemize}
Write $\Gamma = \{\w(M) \mid M \in \mathcal{N}_{\prec_W}(I)\} \subseteq \N_+^r$. For $f \in \mathbb{K}[x_1, \dots, x_n]/I$, denote by $F$ the unique remainder of any polynomial in $f$ after division with $\mathcal{G}$. Then $R/I$ is an order domain with a weight function $\rho : R/I \rightarrow \Gamma \cup \{-\infty\}$ defined by $\rho(0) = -\infty$ and 
\begin{center}
$\rho(f)=\max_{\prec_{\N_+^r}} \{\w(M) \mid M \in \Supp(F)\}$ for $f \neq 0$.
\end{center}
\end{theorem}

If $I$ and $\prec_W$ satisfy the hypotheses of Theorem \ref{theoorderdomain}, we call the pair $(R/I, \prec_W)$ an order domain. 

\begin{example}\label{hermitianpoly}
Let $q$ be a prime power and consider the so-called Hermitian polynomial $x^{q+1} - y^q -y$ and let $I$ be the ideal $I=(x^{q+1} - y^q -y) \subseteq \mathbb{F}_{q^2}[x,y]$, where $\mathbb{F}_{q^2}$ stands for a filed with $q^2$ elements and $q$ is a prime power. We consider $\prec_W$ given by $w(x)=q$, $w(y)=q+1$ and $x \prec_{lex} y$. Then $\mathcal{G} = \{x^{q+1} - y^q - y\}$ is a Groebner basis for $I$, and it is not difficult to verify that $I$ and  $\prec_W$ satisfy the order domain conditions $(C1)$ and $(C2)$, then $(R/I, \prec_w)$ is an order domain.
\end{example}

One main advantage of Theorem \ref{theoorderdomain} is that it allows to construct in a very easy way order domains of higher transcendence degree.

\subsection{Affine-variety codes and order domain codes}
In this section, we present a class of codes, called affine-variety codes and defined in \cite{CGC-cd-art-lax}, obtained by evaluating functions in the coordinate ring of an affine variety on all the $\mathbb{K}$-rational points, i.e. points whose coordinates are in $\mathbb{K}$, of the variety. Let $I$ be any ideal in $R:=\mathbb{K}[x_1, \dots, x_n]$, where $\mathbb{K}:= \mathbb{F}_q$ is the field with $q$ elements. Put
\[
I_q := I + (x_1^q - x_1, x_2^q - x_2, \dots, x_n^q - x_n )
\]

The points of the affine variety $\mathcal{V}(I_q)$ defined by $I_q$ are the $\mathbb{F}_q$-rational points of the affine variety defined by $I$. Let $\mathcal{V}(I_q) = \{ P_1, P_2, \dots, P_N\}$. Since $I_q$ contains the polynomials $x_i^q - x_i$ for all $i= 1,\dots, n$, it is a 0-dimensional and radical ideal. It follows that the coordinate ring
\[
R_q := R / I_q 
\]
of $\mathcal{V}(I_q)$ is an Artinian ring of length $N$ and that there is an isomorphism $\ev$ of $\mathbb{F}_q$-vector spaces
\[
\ev : R_q \ \rightarrow \left(\mathbb{F}_q\right)^N 
\qquad \bar{f} \ \mapsto (f(P_1), \dots, f(P_N))
\]
where $f$ is a representative in $R$ of the residue class $\bar{f}$. Let $L \subseteq R_q$ be an $\mathbb{F}_q$-vector subspace of $R_q$ of dimension $k$. The image $\ev(L)$ of $L$ under the evaluation map $\ev$ is called the \textbf{affine-variety code} and we denote it by $C(I,L)$. The dual code $C^{\perp}(I,L)$ is the orthogonal complement of $C(I,L)$ with respect to the usual inner product on $\mathbb{F}_q^N$. \\

\begin{theorem} \cite{CGC-cd-art-lax}
Every linear code may be represented as an affine-variety code.
\end{theorem}

Let $\mathcal{N}_{\prec_W}(I)$ be the Hilbert staircase of $I$. If $\Supp(b_i) \subseteq \mathcal{N}_{\prec_W}(I_q)$, for all $i = 1,\dots,k$, and  $\lm(b_1) \prec_W \lm(b_2) \prec_W \cdots \prec_W \lm(b_k)$, then the basis $B$ is called \textit{well-behaving} and we define $\mathcal{L}(L) = \{\lm(b_1), \dots, \lm(b_k)\}$. \\
Let $\mathcal{G}$ be a Groebner basis for $I_q$. An ordered pair of monomials $(m_1, m_2)$, with $m_1, m_2$ in $\mathcal{N}_{\prec_W}(I_q)$, is said to be \textit{one-way-well-behaving (OWB)} if for any $f$ such that $\lm(f) = m_1$ and $\Supp(f) \subseteq \mathcal{N}_{\prec_W}(I_q)$, we have
\[
\lm(fm_2 \; \rem \; \mathcal{G}) = \lm(m_1m_2 \; \rem \; \mathcal{G})
\]
where the notation "$f \; \rem \; \mathcal{G}$" stands for the remainder of $f$ modulo $\mathcal{G}$.

\begin{theorem}\label{mindistaffinevarietycodes}
For any monomial ordering $\prec$, the minimum distance of $C(I,L)$ is at least
\[
\min_{p \in \mathcal{L}(L)} \mid \{ s \in \mathcal{N}_{\prec}(I_q) \ \mid \ \exists h \in \mathcal{N}_{\prec}(I_q) \ \text{s.t} \ (p,h) \ \text{is OWB}, \ \lm(ph \; {\rem} \; \mathcal{G})= s \} \mid
\]
and the minimum distance of $C(I,L)^{\perp}$ is at least 
\[
\min_{s \in \mathcal{N}_{\prec}(I_q) \setminus \mathcal{L}(L)} \mid \{p \in \mathcal{N}_{\prec}(I_q) \mid \exists h \in \mathcal{N}_{\prec}(I_q) \ \text{s.t.} \ (p,h) \ \text{is OWB}, \ \lm(ph \; \rem \; \mathcal{G})= s \} \mid
\]
\end{theorem}

\smallskip 

We are going to define the order domain codes, thus we can translate the results on their minimum distance given in Theorem \ref{mindistaffinevarietycodes} into the language of semigroup.

\begin{definition}
Let $(R/I,\prec_W)$ be an order domain and $L \subseteq R_q$, the affine-variety code $C(I,L)$ is called an \textbf{order domain code}.
\end{definition}

\begin{example}\label{hermitiancode}
Let $I = (x^{q+1} + y^q +y) \subseteq \mathbb{F}_{q^2}[x,y]$ and $\prec_W$ be as in Example \ref{hermitianpoly}, with $q=2$. $(R/I, \prec_W)$ is an order domain and then the code from the curve $x^{3} - y^2 -y$ over $\mathbb{F}_{4}$ is an order domain code. \\ We observe that this code is called Hermitian code since it is obtained by evaluating at the points of the Hermitian curve.
\end{example}
Observe that also other algebraic geometric codes, such as norm-trace codes, Reed-Solomon codes and Hyperbolic codes can be put into a form satisfying the order domain conditions (\cite{CGC-cd-book-Geil08}).

\begin{theorem}
Let $(R/I,\prec_W)$ be an order domain and $L \subseteq R_q$. The minimum distance of $C(I,L)$ is at least
\[
\min_{\alpha \in \mathcal{L}(L)} \sigma(w(\alpha))
\]
and the minimum distance of $C(I,L)^{\perp}$ is at least 
\[
\min \{\mu(\w(h)) \mid h \in \mathcal{N}_{\prec}(I_q) \setminus \mathcal{L}(L) \} \geq \min\{\mu(\lambda) \mid \lambda \in \Gamma \setminus \w(\mathcal{L}(L)) \}
\]
and so it is at least 
\begin{equation}\label{bound}
\min\{\mu(\lambda) \mid \lambda \in \Gamma \setminus \w(\mathcal{L}(L)) \}
\end{equation}
where $\Gamma := \w(\mathcal{N}_{\prec}(I))$ is the semigroup of the variable weights, $\mu(\lambda) := \mid\{ \alpha \in \Gamma \mid \exists \beta \in \Gamma \; \text{s.t.} \; \alpha + \beta = \lambda\} \mid$, for $\lambda \in \Gamma$, and $\sigma(\alpha) := \mid \{ \lambda \in \w(\mathcal{N}(I_q)) \mid \exists \beta \in \Gamma \; \text{s.t.} \; \alpha + \beta = \lambda\} \mid$, for $\alpha \in  \w(\mathcal{N}(I_q))$.
\end{theorem}

One of the advantages of the order domain approach to algebraic geometric codes is given by the bound on the distance (\ref{bound}) provided in previous theorem, since this is a bound that can be easily computed from the knowledge of the semigroup.
\subsection{Introduction to Hilbert quasi-polynomials}
In the following we refer to \cite{CGC-alg-book-kreurob05} for standard notation.

Let $I$ be a $W$-homogeneous ideal of $(R,\prec_{W})$. The component of $R/I$ of degree $k\in\mathbb{N}$ is given by
\[
(R/I)_k := \{f \in R/I \mid \w(m) = k\ \ \forall m \in \Supp(f) \} \quad 
\]
The \textbf{Hilbert function} $H_{R/I}^W : \mathbb{N} \rightarrow \mathbb{N}$ of $(R/I,\prec_{W})$ is defined by 
\[
H_{R/I}^W (k) := \dim_{\mathbb{K}}((R/I)_k)
\]
and the Hilbert-Poincar\'e series of $(R/I,\prec_{W})$ is given by
\[
\HP_{R/I}^W (t) := \sum_{k \in \mathbb{N}} H_{R/I}^W(k) t^k \in \mathbb{N}\llbracket t \rrbracket
\]
When the grading given by $W$ is clear from the context, we denote the Hilbert function and the Hilbert-Poincaré series of $(R/I,\prec_{W})$ by $H_{R/I}$ and $\HP_{R/I}$,  respectively.\\
By the Hilbert-Serre theorem, the Hilbert-Poincar\'e series of $(R/I,\prec_{W})$ is a rational function, which is
\[
\HP_{R/I}(t) = \frac{h(t)}{\prod_{i=1}^n( 1-t^{ w_i})}\in \mathbb{Z} \llbracket t \rrbracket
\]

We recall that a function $f \colon \mathbb{N} \to \mathbb{N}$ is a \emph{quasi-polynomial of period s} if there exists a set of $s$ polynomials $\{p_0, \dots, p_{s-1}\}$ in $ \mathbb{Q}[x]$ such that $f(n) = p_i(n)$ when $n \equiv i \bmod s$. Let $d:= \lcm(w_1,\dots, w_n)$ and let $(R/I,\prec_{W})$ be as above. We now refer to \cite{CGC-alg-book-vasconcelos04} for Hilbert quasi-polynomial theory. There exists a unique quasi-polynomial $P_{R/I}^W:=\{P_0, \dots, P_{d-1}\}$ of period $d$ such that $H_{R/I}(k) = P_{R/I}^W(k)$ for all $k$ sufficiently large (that we denote with $k\gg 0$), i.e.
\[
H_{R/I}(k) = P_i(k) \qquad \forall i \equiv k\mod d  \quad \text{and} \quad \forall k\gg 0
\]
$P_{R/I}^W$ is called the \textit{Hilbert quasi-polynomial associated to} $(R/I,\prec_{W})$. Observe that if $d=1$, then $P_{R/I}^W$ is a polynomial, and it is simply the \textit{Hilbert polynomial}. We underline that the Hilbert quasi-polynomial does not depend on the chosen monomial ordering $\prec_{\mathcal{M}}$, but only on the variable weights. It consists of $d$ polynomials, which are not necessarily  distinct. 
The minimum integer $k_0 \in \mathbb{N}$ such that $H_{R/I}(k) = P_{R/I}^W(k) \; \forall \ k \geq k_0$ is called \textit{generalized regularity index} and we denote it by $\ri(R/I)$.  
All the polynomials of the Hilbert quasi-polynomial $P_{R/I}^W$ have rational coefficients and the same degree $r \leq n-1$, where the equality holds if and only if $I=(0)$. In this latter case, the leading coefficient $a_{n-1}$ is the same for all $P_i$, with $i=0, \dots, d-1$, and $a_{n-1} = \frac{1}{(n-1)! \prod_{i=1}^n w_i}$.

\section{Computational improvements for Hilbert quasi-polynomials, with an application to order domains} \label{appl}

In this section we present an algorithm for an effective computation of Hilbert quasi-polynomials and we show how to exploit them for checking order domain conditions. \\

We have improved the Singular procedures showed in \cite{CGC-alg-art-cabmasc16} to compute the Hilbert quasi-polynomial for rings $\mathbb{K}[x_1,\ldots,x_n]/I$.\\
Before showing the algorithm and our improvement, we give some preliminary results. 

\begin{theorem}\label{theo1stanley} \textnormal{\cite{GCG-alg-book-stanley2007}}
Let $\alpha_1, \dots, \alpha_d$ be a fixed sequence of complex numbers, $d \geq 1$ and $\alpha_d \neq 0$. The following condition on a function $f: \mathbb{N} \rightarrow \mathbb{C}$ are equivalent:
\begin{itemize}
\item[(i)] \[
\sum_{n\geq 0} f(n) x^n = \frac{P(x)}{Q(x)} ,
\]
where $Q(x) = 1+\alpha_1 x + \cdots +\alpha_d x^d$ and $P(x)$ is a polynomial in $x$ of degree less than $d$.
\item[(ii)] For all $n \geq 0$,
\begin{equation}\label{eq1theostan}
f(n+d) + \alpha_1 f(n+d-1) + \cdots + \alpha_d f(n)=0.
\end{equation}
\item[(iii)] For all $n \geq 0$,
\begin{equation}\label{eq2theostan}
f(n) = \sum_{i=1}^k P_i(n)\gamma_i^n ,
\end{equation}
where $1+\alpha_1x + \cdots + \alpha_dx^d = \prod_{i=1}^k (1-\gamma_i x)^{d_i}$, the $\gamma_i$ 's are distinct, and $P_i(n)$ is a polynomial in $n$ of degree less than $d_i$. 
\end{itemize}
\end{theorem}

\begin{proposition}\label{prop2stanley} \textnormal{\cite{GCG-alg-book-stanley2007}}
Let $f: \mathbb{N} \rightarrow \mathbb{C}$ and suppose that
\[
\sum_{n\geq 0} f(n)x^n = \frac{P(x)}{Q(x)}
\]
where $P,Q \in \mathbb{C}[x]$. then there is a unique finite set $E_f \subset \mathbb{N}$ and a unique function $f_1: E_f \rightarrow \mathbb{C}* = \mathbb{C} \setminus \{0\}$ such that the function $g: \mathbb{N} \rightarrow \mathbb{C}$  defined by
\[
g(n) =
\begin{cases}
f(n) & \text{ if } n \not \in E_f, \\
f(n)+f_1(n) & \text{ if } n \in E_f
\end{cases}
\]
satisfies $\sum_{n \geq 0} g(n) x^n = R(x) / Q(x)$, where $R \in \mathbb{C}[x]$ and $\deg R < \deg Q$. Moreover, assuming $E_f \neq \emptyset$ (i.e. $\deg P \geq \deg Q$), define $\mathrm{m}(f) = \max \{ i : i \in E_f\}$. Then:
\begin{itemize}
\item[(i)] $\mathrm{m}(f) = \deg P - \deg Q$,
\item[(ii)] $\mathrm{m}(f)$ is the largest integer $n$ for which (\ref{eq1theostan}) fails to hold,
\item[(iii)] writing $Q(x) = \prod_{i=1}^k (1-\gamma_i x)^{d_i}$ as above, there are unique polynomials $P_1, \dots, P_k$ for which  (\ref{eq2theostan}) holds for $n$ sufficiently large. Then $\mathrm{m}(f)$ is the largest integer $n$ for which  (\ref{eq2theostan}) fails.
\end{itemize}
\end{proposition}

Thanks to these two results, we are able to compute the generalized regularity index of $R/I$, one compute the Hilbert-Poincaré series of $R/I$.

\begin{proposition}\label{RI}
Let $(R/I,\prec_W)$ be as usual and let $\HP_{R/I}^W (t) = \frac{h(t)}{g(t)}$. Then the generalized regularity index of $R/I$ is given by
\[
\ri(R/I) = \max \{0, \deg h(t)  - \deg g(t) +1 \}
\]
\end{proposition}

\begin{proof}
By Theorem \ref{theo1stanley} and Proposition \ref{prop2stanley}, and since
\[
\sum_{k \geq 0} H_{R/I} (k) t^k = \frac{h(t)}{g(t)}
\]
with $g(t) = \prod_{i=1}^n (1- t^{w_i}) = \prod_{j=0}^{d-1} (1- \zeta^j t)^{\alpha_j}$, where $\zeta$ is a primitive $d$th root of unity and $\sum_{i=1}^n w_i=\sum_{j=0}^{d-1}\alpha_j=\deg g(t)$, we obtain that, for all $k \geq k_0$ with
\[
k_0 = \begin{cases}
0& \text{if } \deg h(t) < \deg g(t) \\
\deg h(t)  - \deg g(t) +1& \text{if } \deg h(t) \geq \deg g(t)
\end{cases}
\quad ,
\]
the Hilbert function can be written as
\[
H_{R/I}(k) = \sum_{i=0}^{d-1} S_i (k) \zeta^{ik}
\]
where $S_i(k)$ is a polynomial in $k$ of degree less than $\alpha_i$. Then, by uniqueness of the Hilbert quasi-polynomial of period $d$, we deduce that the $i$th polynomial of $P_{R/I}$ is given by 
$P_i(t) := \sum_{j=0}^{d-1} S_j(t)\zeta^{ij}$ and that $H_{R/I}(k) = P_i(k)$ for $k \geq k_0 \text{ and } k \equiv i \mod d$.
\end{proof}

\begin{remark}
Since $\HP^W_R(t) = \frac{1}{\prod_{i=1}^{n} (1- t^{w_i})}$, the generalized regularity index of $R$ is 0.
\end{remark}

Thanks to the following two results we can recover $H^W_{R/I}(k)$ from $H^W_{R}(k)$ and $P_{R/I}^{W'}$ from $P_R^W$, where $W$ is obtained by $W'$ dividing each $w_i \in W'$ by $\gcd(W')$.

\begin{proposition}\label{recoverH}
Let $\bar{I}$ be the initial ideal of $I$ and $\HP^W_{R/\bar{I}}(t) = \frac{h(t)}{g(t)}$ the Hilbert-Poincaré series of $R/\bar{I}$, with $h(t) = \sum_{i=0}^s h_i t^i$. Then 
\[
H^W_{R/\bar{I}}(k) = \sum_{i=0}^s h_i H^W_R(k-i) 
\]
for all $k \geq 0$, with $H^W_{R}(k) = 0$ for all $k <0$.
\end{proposition}

\begin{proof}
Since \[
\HP_R(t) = \sum_{k \geq 0} H_R(k)t^k = \frac{1}{\prod_{i=1}^n (1-t^{w_i})}
\]
we have
\[
\HP_{R/\bar{I}}(t) = \sum_{k \geq 0} H_{R/\bar{I}}(k)t^k = \frac{h(t)}{\prod_{i=1}^n (1-t^{w_i})} = \left( \sum_{k \geq 0} H_R(k)t^k \right) \left( \sum_{j=0}^s h_jt^j \right) = \sum_{k \geq 0} \left( \sum_{j=0}^s h_j H_R(k-j) \right) t^k
\]
where $H_R(k) = 0$ for all $k < 0$. Therefore, $H_{R/\bar{I}}(k) = \sum_{j=0}^s h_j H_R(k-j)$, for all $k \geq 0$.
\end{proof}

\begin{lemma}\label{PRW} (\cite{CGC-alg-art-cabmasc16})
Let $W' := a \cdot W = [w_1', \dots, w_n']$ for 
some $a \in \mathbb{N}_+$ and let $\displaystyle \HP^W_{R/I}(t) = \frac{\sum_{j=0}^s h_jt^j}{\prod_{i=1}^n(1-t^{w_i})}$. Then it holds:
\begin{enumerate}

\item 
$P_{R/I}^W(k) = \sum_{j=0}^s h_j P_R^W(k-j)$ for all $k \geq \ri(R/I)$

\item 

$P_R^{W'}=\{P_0', \dots, P_{ad-1}'\}$ is such that
\[
P_i'(x) =\begin{cases}
0&\text{ if } a \nmid i \; \\
P_{\frac{i}{a}}\left(\frac{x}{a}\right)& \text{ if } a \mid i \;  
\end{cases}
\]

\end{enumerate}
\end{lemma}

\subsection{Algorithm for computing Hilbert quasi-polynomials}\label{algquasipoly}
Let $(R/I, \prec_W)$ be as usual, we wish to compute its Hilbert quasi-polynomial $P_{R/I}^W :=\{ P_0, \dots, P_{d-1}\}$. Since we know some degree bounds for Hilbert quasi-polynomials, we can compute them by means of interpolation.
\smallskip
 
First of all, let us consider $I = (0)$ and $W$ such that the $w_i$'s have no common factor. 
Each $P_j$ has degree equal to $n-1$, so, given $j=0, \dots, d-1$, we want to calculate $P_j (x) := a_0 + a_1 x + \dots + a_{n-1} x^{n-1}$
such that
\[
P_j (k) = H_R(k) \qquad \forall k \geq 0 \quad \text{and} \quad k \equiv j  \mod d
\]

Therefore, let us consider the first $n$ positive integers $x_r$ such that $P_j (x_r) = H_R (x_r)$
\[
x_r := j + rd, \quad \text{for } r = 0, \dots, n-1
\]
By construction, the polynomial $P_j(x)$ interpolates the points $(x_r, H_R(x_r))$. 

Since we know the leading coefficient $a_{n-1}$, we can reduce the number of interpolation points $x_r$, and we get a system of linear equations in the coefficients $a_i$, with $i=0, \dots, n-2$.  The system in matrix-vector form reads
\
\begin{equation}\label{linearsystem}
\begin{bmatrix}
1 & x_0 & \ldots  & x_0^{n-2} \\
1 & x_1 & \ldots   & x_1^{n-2} \\
\vdots & \vdots    & \vdots    & \vdots \\
1 & x_{n-2} & \ldots  & x_{n-2}^{n-2} 
\end{bmatrix}
\begin{bmatrix} a_0 \\ a_1 \\ \vdots \\ a_{n-2} \end{bmatrix}  = \begin{bmatrix} H_R(x_0) - a_{n-1}x_0^{n-1} \\ H_R(x_1) - a_{n-1}x_1^{n-1} \\ \vdots \\ H_R(x_{n-2}) - a_{n-1}x_{n-2}^{n-1}
\end{bmatrix}
\end{equation}

We observe that for computing $P_j$ the algorithm requires the computation of $n-1$ values of the Hilbert function, the construction of a Vandermonde matrix of dimension $n-1$ and its inversion. We have not yet shown how to calculate $H_R(x_r)$, for $r=0, \dots, n-2$.
 
\smallskip
 
Let $k \in \mathbb{N}$. The problem of calculating $H_R(k)$ is equivalent to the problem of determining the number of partitions of an integer into elements of a finite set $S := \{ w_1, \dots, w_n \}$, that is, the number of solutions in non-negative integers, $\alpha_1, \dots, \alpha_n$, of the equation
\[
\alpha_1 w_1 + \cdots + \alpha_n w_n =k
\]
This problem was solved in  \cite{CGC-alg-art-sylvester1882}, \cite{CGC-alg-art-glaisher1909} and the solution is the coefficient of $x^k$ in the following power series
\begin{equation}\label{power series}
\frac{1}{ (1- x^{w_1}) \cdots (1-x^{w_n})}
\end{equation}

We are going to give an efficient method for getting the coefficient of $x^k$ in the power series expansion of Equation (\ref{power series}). We refer to \cite{CGC-alg-art-lee92} for an in-depth analysis on the power series expansion of a rational function. Let

\[
g(x) = \prod_{i=1}^k (1-\lambda_i x)^{\alpha_i} \qquad \text{and} \qquad f(x) = \prod_{i = k+1}^l (1- \lambda_i x)^{\alpha_i}
\]
be  polynomials in $\mathbb{C}[x]$, where $\alpha_1,\ldots, \alpha_l$ are non-negative integers, all $\lambda_i$'s are distinct and non-zero and the degree of $f(x)$ is less than that of $g(x)$.

\begin{lemma}\label{nbn}
Let 
\[
\frac{f(x)}{g(x)} = \sum_{n \geq 0} b(n) x^n
\]
the power series expansion of $f(x)/g(x)$. Then,
\[
b(n)n = \sum_{r=1}^n \left(\sum_{i=1} ^k \alpha_i \lambda_i^r - \sum_{i=k+1}^l \alpha_i \lambda_i^r \right) b(n-r)
\]
\end{lemma}





\medskip 

Let $\zeta := \zeta_d$ be a primitive $d$th root of unity. Since
\[
\prod_{i=1}^n (1-x^{w_i}) =  \prod_{j \in T_n} (1-\zeta^j x)^n \prod_{j \in T_{n-1}} (1-\zeta^j x)^{n-1} \cdots \prod_{j \in T_1} (1-\zeta^j x) 
\]

for some pairwise disjoint subsets $T_1, \dots, T_n \subseteq \{0,\dots, d-1\}$, we can apply Lemma \ref{nbn} with $g(x) :=\prod_{i=1}^k (1-x^{w_i}) $ and $f(x) = 1$. 
Observing that any $j$ in $\{0,\ldots,d-1\}$ appears in exactly one of the $T_i$'s, say $T_\iota$, and that $\zeta^j$ appears $\iota$ times in the product,  we obtain the following recursive formula for computing $H_R(k)$

\begin{equation}\label{H_R^W recursive}
H_{R}(k) = \frac{1}{k} \sum_{r=1}^k \left[ \sum_{i=1}^n i \left( \sum_{j \in T_i} \zeta^{jr} \right) \right] H_{R}(k-r)
\end{equation}

It follows that if we know $H_R(i)$ for all $i=1, \dots, k-1$, we can easily compute $H_R(k)$ by means of (\ref{H_R^W recursive}). 

\medskip
 
Given an equation $\alpha_1 w_1 + \dots + \alpha_n w_n = k$ the problem of counting the number of non-negative integer solutions $\alpha_1, \dots, \alpha_n$  could be solved also using brute force. But, given $k \in \mathbb{N}$, to compute $H_R^W(k)$ with brute force needs $O(k^n)$ operations, whereas the procedure which we have implemented has a quadratic cost in $k$, in fact it needs $O(n k^2)$ operations (\cite{CGC-alg-art-cabmasc16}).

\medskip

Up to now we have shown how to calculate $P_R^W$. For computing a Hilbert quasi-polynomial $P_{R/I}^W = \{P_0, \dots, P_{d-1}\}$, for any vector $W$ and any homogeneous ideal $I$ of $R$, the procedure computes first $P_R^{W'}$, where $W'$ is obtained by dividing $W$ by $\gcd(w_1, \dots, w_n)$, and then it produces $P_{R/I}^W$ starting from $P_R^{W}$, using the relation between $P_R^{W'}$ and $P_{R/I}^W$ showed in Lemma \ref{PRW}.b. \\

We reduced the complexity of the algorithm presented in \cite{CGC-alg-art-cabmasc16} by exploiting formulas for the 2th and 3rd coefficient of the Hilbert quasi-polynomial of $R$, whenever possible, and by computing the values of the Hilbert function once, instead of $d$ times.
In particular, the considered formulas, which are defined in \cite{CGC-alg-art-cabmasc16}, allow to reduce the dimension of the linear system (\ref{linearsystem}) to be solved in order to compute the Hilbert quasi-polynomials $P_R^W$. We compare experimentaly the difference of speed of the two algoritms for computing $H_R^W$, that in \cite{CGC-alg-art-cabmasc16} and ours, with $R =\mathbb{K}[x_1, \dots, x_n]$ where $n \in \{2, \dots, 5\}$ and $W$ a random vector of weights $w_i \in [1, \dots, 12]$.
The following table summarizes the results obtained using an Intel(R) Core(TM) i3-6100 CPU @ 3.70 GHz processor and 8 GB of RAM. 
\vspace{0.5cm}
\begin{center}
    \begin{tabular}{ | l | l | l | p{5cm} |}
    \hline
    n & W & Old Algorithm Time in ms & New Algorithm Time in ms \\ \hline
    2 & [1,3] & 20 & 0\\ \hline
    3 & [2,5,12] & 3300 & 580 \\ \hline
    4 & [1,4,5,8] & 1080 & 140 \\ \hline
    5 & [2,2,6,9,12] & 5340 & 230 \\ \hline
    \end{tabular}
\end{center}
\vspace{0.5cm}

For $n\geq 6$, we have to choose suitably low weights $w_i$, otherwise we encounter software limitation of Singular. For example, for $n=6$ we set $W=[1,1,1,2,2,9]$, and the old algorithm obtains some \textit{Int-overflow errors}, and then it computes wrongly 4 of the 18 polynomials of $H_R^W$, in 1050 ms, whereas our algorithm computes correctly all the 18 polynomials in 100 ms.

\subsection{Hilbert quasi-polynomials for order domain codes}
To test if a pair $(R/I,\prec_W~)$ satisfies the order domain condition $(C1)$, it suffices to compute a Groebner basis of $I$ w.r.t. $\prec_W$ and, for each polynomial in the basis, to check the two monomials of highest weight. Whereas, for the condition $(C2)$ we would need to study the ideal $I$ or the semigroup $\mathcal{N}_{\prec_W}(I)$. We give an alternative and efficient way to establish if the condition $(C2)$ is respected by $I$ and $\prec_W$. \\

\begin{theorem}
Let $(R/I,\prec_W)$ be as usual. The following are equivalent:
\begin{itemize}
\item[(i)] $I$ and $\prec_W$ satisfy the order domain condition $(C2)$;
\item[(ii)] $H_{R/\bar{I}}(k) \in \{0,1\}$ for all $0 \leq k < \ri(R/\bar{I})$ and each $P_i$ in $P_{R/\bar{I}}^W$ is the constant polynomial 0 or 1.
\end{itemize}
\end{theorem}

\begin{proof}
Since $\mathcal{N}_{\prec_W}(I) = \{ \lm(f) \mid f \in R/\bar{I}\}$, then $H_{R/\bar{I}}(k) = \mid \{m \in \mathcal{N}_{\prec_W}(I) \mid \w(m)= k \} \mid$. We observe that condition $(C2)$ can be equivalently formulated
as requiring that for any possible weight there is at most one monomial in $\mathcal{N}_{\prec_W}(I)$ enjoying
that weight. In terms of Hilbert function, this conditions is then equivalent to
\[
H_{R/\bar{I}}(k) \in \{0,1\}\mbox{ for all }k \geq 0.
\]
Recalling that $H_{R/\bar{I}}$ will be eventually equal to the quasi-polynomial $P_{R/\bar{I}}^W = \{ P_0, \dots, P_{d-1}\}$, our assertion follows.
\end{proof}

\begin{remark}
The only requirement that each $P_i$ in $P_{R/\bar{I}}^W$ is the constant polynomial 0 or 1 does not imply $(i)$, as we now show.  Let $R= \mathbb{K}[x_1, x_2]$ with standard grading, and $I = (x_1^2, x_1x_2) \subseteq R$. The Hilbert polynomial is the constant polynomial 1 but $H_{R/I}(1) = 2$, then $I$ does not satisfy $(C2)$. That is because $H_{R/\bar{I}}(k) = P_{R/\bar{I}}^W (k)$ for all $k \geq \ri(R/\bar{I})$, but $P_{R/\bar{I}}^W$ does not give any information for $k < \ri(R/\bar{I})$. \\
\end{remark}

\begin{corollary} \label{cor1}
Let $(R/I,\prec_W)$ be as usual, $\bar{I} =\mathrm{in}(I)$ and $\mathcal{G}$ a Groebner basis for $I$. If 
\begin{itemize}
\item any $g \in \mathcal{G}$ has exactly two monomials of highest weight in its support, and
\item $H_{R/\bar{I}}(k) \in \{0,1\}$ for all $0 \leq k < \ri(R/\bar{I})$ and each $P_i$ in $P_{R/\bar{I}}^W$ is the constant polynomial 0 or 1.
\end{itemize}
then $(R/I,\prec_W)$ is an order domain.
\end{corollary}

\begin{example}
Let  $(R/I,\prec_W)$ be as in Example \ref{hermitiancode}. The Hilbert-Poincaré series of $(R/\bar{I}, \prec_W)$ is given by $\HP_{R/\bar{I}}(t) = \frac{1-t^6}{(1-t^2)(1-t^3)}$, then $\ri(R/\bar{I}) = 2$. The Hilbert quasi-polynomial is $P_{R/\bar{I}}= \{ P_0, \dots, P_5\}$, with $P_i = 1$ for all $i=0, \dots, 5$. Since $H_{R/\bar{I}} (k) = P_{R/\bar{I}}(k)$ for all $k \geq 2$, we have only to check $H_{R/\bar{I}}(1)$ which is obviously equal to 0, since $1$ is a gap in the semigroup $\Gamma=<2,3>$.
\end{example}

Now  we are ready to describe the following 
\begin{algorithm}[Check Order Domain] \label{algo1}
\ 
\begin{itemize}
\item Input: $(R/I, \prec_W)$.
\item Output: IsOrderDomain $\in \{$True, False$\}$.
\end{itemize}
\
\begin{enumerate}
\item IsOrderDomain:= False;
\item Compute a Groebner basis $\mathcal{G}$ for $I$;
\item If $\w(\lt(g))=\w(\lt(g-\lt(g)))$ for all $g \in \mathcal{G}$, then \label{C1}
\begin{enumerate}
\item Let $ k_1 := \max \{ \ri(R/\bar{I}), d(n-1)\}$; \label{k1}
\item Compute $H_{R}(k)$ for all $1 \leq k < k_1$;\label{Hk}
\item If $H_{R/\bar{I}}(k) \in \{0,1\}$ for all $1 \leq k < k_1$ and each $P_i \in P_{R/\bar{I}}$ is $0$ or $1$, then \label{Test}
\begin{itemize}
\item[$\ $] IsOrderDomain:=True.
\end{itemize}
\end{enumerate}

\item Return IsOrderDomain. 
\end{enumerate}
\end{algorithm}

We recall the results necessary to the description of the algorithm.
 
\begin{itemize}
 \item Line \ref{C1}. We check the condition $(C1)$ of Corollary \ref{cor1}. If it is satisfied we test the second one.
 \item Line \ref{k1}-\ref{Hk}. We compute the first $k_1$ values of the Hilbert function of the polynomial ring $R$  using Equation (\ref{H_R^W recursive}), where $k_1$ is the maximum between $d(n-1)$, that is the number of interpolation points, and the regularity index, $\ri(R/\bar{I})$, which is known thanks to Proposition \ref{RI}, where $\deg h(t)$ is computed by Singular.
\item Line \ref{Test}. Thanks to Proposition \ref{recoverH} we are able to check if $H_{R/\bar{I}}(k) \in \{0,1\}$ for the values $0 \leq k < \ri(R/\bar{I})$. If the previous test does not fail we can compute, using the algorithm showed in Section \ref{algquasipoly}, the quasi-polynomial of $R/\bar{I}$, completing the test of the second condition described in Corollary \ref{cor1}.
\end{itemize}

We would like to point out an interesting aspect of our approach. For our computation of Hilbert quasi-polynomials it is essential to work in a characteristic-0 field. In our applications, e.g. to coding theory, we can have positive characteristic fields for the order domain. However, the actual field matters only for the computation of the Groebner basis of $I$. Once the Groebner basis has been obtained, the monomials under the Hilbert staircase are the same for any field and so we can consider the leading monomials of the obtained Groebner basis  as if they were on a characteristic-zero field.

\section{Applications to codes from maximal curves} \label{maximal}
A maximal curve over a finite field $\mathbb{F}_q$ is a projective geometrically irreducible non-singular algebraic curve defined over $\mathbb{F}_q$ whose number of $\mathbb{F}_q$-rational points attains
the Hasse-Weil upper bound
\[
q + 1 + 2g\sqrt{q} ;
\]
where $g$ is the genus of the curve. Maximal curves, especially those having large genus with respect to $q$, are known to be very useful in coding theory. In this section, we show some examples of affine-variety codes constructed over maximal curves, which are also order domains.

\begin{example}
Let $\chi\subseteq \mathbb{P}^3$ is a $\mathbb{F}_{49}$-maximal curve of genus $g = 7$ (\cite{CGC-alg-art-fangiupla12})  whose affine plane curve is
\[
 y^{16} = x(x + 1)^6
\]
Let $I = ( y^{16} - x(x + 1)^6 ) \subseteq \mathbb{F}_{49}[x,y]$ and let $\prec_W$ be given by $\w(x) = 16, \w(y)=7$ and $x \prec_{lex} y$. It is easy to verify that $I$ and $\prec_W$ satisfy the order domain condition $(C1)$. Let us check condition $(C2)$. Obviously, $\bar{I} = (y^{16})$. With our algorithm we have computed the Hilbert quasi-polynomial $P_{R/\bar{I}}= \{P_0, \dots, P_{111}\}$,
and each $P_i$ is equal to 1. With Singular we have computed $\HP_{R/\bar{I}}(t)= \frac{1-t^{112}}{(1-t^7)(1-t^{16})}$ and so $\ri(R/\bar{I}) = 90$, which means that we are left with computing $H_{R/\bar{I}}(k)$, with $0 < k < 90$. With our recursive algorithm
we can easily see that all obtained values are in $\{0,1\}$. Then we can conclude that $(R/I, \prec_W)$ is an order domain.  
\end{example}
In the previous example we could avoid to compute the Hilbert function for values less than 90, thanks to the following lemma.
\begin{proposition}\label{purepower}
Let $I \subseteq \mathbb{K}[x_1, \dots, x_n]$ be an ideal such that its initial ideal $\bar{I} \subseteq \mathbb{K}[x_1, \dots, x_{n-1}]$, up to a reordering of variables. If each polynomial in the Hilbert quasi-polynomial of $R/\bar{I}$ is 0 or 1, then $H_{R/\bar{I}}(k) \in \{0,1\}$ for all $k \geq 0$.
\end{proposition}
\begin{proof}
Suppose by contradiction $H_{R/\bar{I}}(\tilde{k}) \not \in \{0,1\}$ for some $\tilde{k} \geq 0$, then there exist two distinct monomials $m_1, m_2 \in R/\bar{I}$ with the same weight $\tilde{k}$. For all $\alpha_n \geq 0$ also the monomials $m_1 x_n^{\alpha_n}, m_2x_n^{\alpha_n} \in R/\bar{I}$, and they are distinct with the same weight $\tilde{k}+\alpha_n w_n$. In particular, it holds also for $\tilde{k}+\alpha_n w_n  \geq \ri(R/\bar{I})$, but this contradicts our hypothesis on the Hilbert quasi-polynomial of $R/\bar{I}$.
\end{proof}

\begin{remark}
If $I$ and $\prec_W$ satisfy the hypothesis of Proposition \ref{purepower}, then $(R/I, \prec_W)$ is an order domain. 
\end{remark}

\begin{example}
Let $q>2$ be a prime power. The GK-curve, introduced by Giulietti and Korchmaros in \cite{GCG-alg-art-giulkorch09}, is the curve $\mathcal{C}_3 \subseteq \mathbb{P}^3$ defined over $\mathbb{F}_{q^6}$ by the affine equations
\[
v^{q+1} = u^q + u \quad \text{ and } \quad w^{\frac{q^3+1}{q+1}} = vh(u)
\]
where $h(u) = (u^q + u)^{q-1}-1$. This curve is maximal over $\mathbb{F}_{q^6}$ and it is so far the only known example of a maximal curve which cannot be dominated by the Hermitian curve. It turns out in \cite{CGC-alg-art-garstich10} that $\mathcal{C}_3$ can also be defined by the equations
\[
v^{q+1} = u^q + u \quad \text{ and } \quad w^{\frac{q^3+1}{q+1}} = v^{q^2}-v
\]
We are going to investigate if $\mathcal{C}_3$, with an opportune generalized weighted degree ordering $\prec_W$, defines an order domain. Let $q=3$, and $I=(v^4 - u^3 -u, w^7 - v^9 +v) \in \mathbb{F}_{3^6}[u,v,w]$ with $u \prec_{lex} v \prec_{lex} w$. Obviously, $\mathcal{G} = \{ v^4 - u^3 -u, w^7 - v^9 +v\}$ is a Groebner basis for $I$ and $\bar{I}= (v^4, w^7)$. In order to satisfy condition $(C1)$, we set $W=[28,21,27]$. Note that, up to a constant factor, $W$ is unique. Since $\lcm(28,21,27)=756$, we have computed the Hilbert quasi-polynomial $P_{R/\bar{I}} = \{P_0, \dots, P_{755} \}$ of $(R/\bar{I}, \prec_W)$. Since each $P_i$ turns out to be $1$, thanks to the remark above, we can conclude that $(R/I, \prec_W)$ is an order domain. 
\end{example}

Let us show a last example for which the condition $(C2)$ does not hold.
\begin{example}
Let $\mathcal{R}(\ell)$ be the \textit{Ree curve} defined by
\[
\mathcal{R}(\ell) : 
\begin{cases}
y^{\ell} - y = x^{\ell_0}(x^{\ell} - x) \\
z^{\ell} - z = x^{\ell_0}(y^{\ell} -y)
\end{cases}
\quad \text{with } \ell_0 = 3^r, \; r \geq 0, \; \ell = 3\ell_0^2.
\]
$\mathcal{R}(\ell)$ is $\mathbb{F}_q$-maximal for $q = \ell^6$.
Consider $r=0$ and $\ell_0=1$, then $I = (x^4 - x^2 -y^3 + y, xy^3 -xy -z^3 + z) \in R= \mathbb{F}_{3^6}[x,y,z]$ with $ z \prec_{lex} y \prec_{lex} x$. The set $\mathcal{G}= \{x^4 - x^2 -y^3 + y, xy^3 -xy -z^3 + z\}$ is a Groebner basis for $I$ and $\bar{I} = (x^4, xy^3)$. In order to verify condition $(C1)$, we choose as generalized weighted degree ordering $\prec_W$ that defined by $W = [3,4,5]$. The computation of the Hilbert quasi-polynomial for $(R/\bar{I}, \prec_W)$ gives 60 distinct polynomials of degree 1, it means that $(R/I, \prec_W)$ is not an order domain.  
\end{example}

\section{Conclusions and further comments} \label{concl}

Algorithm \ref{algo1} allows to decide effectively if a quotient ring can be seen as an order domain w.r.t. a generalized weighted degree ordering. It can be applied to coding theory, and it can indeed solve some interesting examples, as  we have shown, and as such we believe it can become a convenient tools for coding theorists working with algebraic geometric codes.

We see at least two paths to follow in order to increase the impact of our approach, as we elaborate below.

\noindent 
An advantage of order domain codes (with respect to more traditional codes over curves) is that we can
      build them on higher dimensions variety, e.g. the surface in Example 51 \cite{CGC-cd-art-AndeGeil}. The higher dimension requires
      generalized weighted degree orderings $\prec_W$, with $W \in (\mathbb{N}_+^r)^n$ and $r\geq 2$, that cannot be tackled by our present theory and algorithms. Therefore, we believe that this extension is natural and worth studying.
    
\noindent      
In addition, our algorithm is well-suited when a given variety has been chosen to build the code. Often it happens that we know an infinite family of varieties that would be ideal to build codes on them, e.g. known families of maximal curves. In this situation we would need to adapt our computational approach to a more theoretical one,
in order to use the core ideas of our algorithms for generating general proofs.

\section*{Acknowledgements} These results are included in the first author's PhD thesis and so she would like to thank her supervisors: the other two authors. The authors would like to thank the referee for valuable suggestions.

\medskip

Received October 2016; revised 16 November 2016.

\medskip

\end{document}